\RequirePackage{fix-cm}
\documentclass[smallextended]{svjour3}       
\smartqed  
\usepackage{graphicx}
\usepackage{amsmath,amsfonts}
\DeclareMathOperator{\spn}{span}

\newtheorem{thm}{Theorem}[section]
 
 \newtheorem{lem}[thm]{Lemma}
 \newtheorem{prop}[thm]{Proposition}
 \numberwithin{equation}{section}
\usepackage{color}

\begin{document}

\title{Sharp estimates for the covering numbers of the Weierstrass fractal kernel
}


\author{D. Azevedo \and
        K. Gonzalez \and
        T. Jord\~{a}o
}


\institute{D. Azevedo \at
\email{douglasa@utfpr.edu.br}\\
DAMAT- Universidade Tecnológica Federal do Paraná, Brazil.                     
\and
K. Gonzalez \at
\email{karina.navarro@correounivalle.edu.co}\\ ICMC - Universidade de S\~{a}o Paulo, Brazil.
\and
T. Jord\~{a}o \at
\email{tjordao@icmc.usp.br}\\
Tel.: +55-16-3373-713\\ICMC - Universidade de S\~{a}o Paulo.\\
Av. Trabalhador saocarlense, 400, São Carlos, S\~{a}o Paulo, Brazil, 13566-590. 
}

\date{Received: date / Accepted: date}

\maketitle

\begin{abstract}
In this paper we use the infamous continuous and nowhere differentiable Weierstrass function as a prototype to define a ``Weierstrass fractal kernel". We  investigate  properties of the reproducing kernel Hilbert space (RKHS) associated to this kernel by  presenting an explicit  characterization of this space. In particular, we show that this space has a dense subset composed of continuous but nowhere  differentiable functions. Moreover, we present sharp estimates for the covering numbers of the unit ball of this space as a subset of the continuous functions.
\keywords{Covering numbers \and Weierstrass fractal kernel \and Fourier series expansion \and reproducing kernel Hilbert space \and continuous nowhere differentiable functions.}
\subclass{Primary 47B06 - 42A16\and Secondary 46E22 - 26A15 - 42A32.}
\end{abstract}

\section{Introduction}
\label{intro}
In 1872 K. Weierstrass presented a particular class of trigonometric series as a collection of continuous but nowhere differentiable functions (CNDF).\ These functions are defined in terms of the following Fourier series expansion
\begin{equation}\label{1.1} 
w_{a,b} (x)= \sum_{n=0}^{\infty} a^n \cos(b^n \pi  \,x),\,\,\,x\in\mathbb{R}.
\end{equation}
 For $0<a<1$ it is clear that it defines a continuous bounded function.\ Under this assumption, Weierstrass proved  that $w_{a,b}$ is nowhere differentiable provided that $ab\geq 1+\frac{3\pi}{2}$,  with $b$ an odd integer (\cite{JP,johsen}).\ G.H. Hardy relaxed this condition in \cite{hardy} by showing that  for $ab\geq 1$, the Weierstrass function is nowhere differentiable.

The Weierstrass function is an inspiration for the Weierstrass-Mandelbrot curve, a fractal curve widely explored in fractal geometry (\cite{BL}).\ It is well known that the graph of $w_{a,b}$ is also a fractal curve.\ In particular, for every $b\geq 2$ integer, there exists  $a_b$ in the open interval $(1/b, 1)$ such that the Hausdorff dimension of the graph of $w_{a,b}$ is equal to $D=2+ \log(a)/\log(b)$, for every $a\in (a_b,1)$ (see \cite{BBR}).\ Many examples of CNDF are known, mainly in fractal theory and applications, and the Weierstrass function is definitely one of the most prominent.\ The function $M:\mathbb{R}\to \mathbb{C}$
given by
\begin{equation}\label{Mand}
    M(x)= \sum_{n=0}^{\infty}r^{-Kn}e^ {ir^{n}x},
    \,\,x\in \mathbb{R},
\end{equation}
where $r>1$ and $0<K<1$, originally investigated  by Mandelbrot (\cite[p.388]{Mand}) 
is an  object of interest in the study of fractional Brownian motion (FBM) which is intermediated by  a stochastic fractal function characterized by its statistical behavior.\ For $0<a<1$ and $ab\geq1$, the function
\begin{equation}\label{Wcomplex}
    w_{a,b,\mathbb{C}}(x)=\sum_{n=0}^{\infty}a^{n}e^ {ib^n x}, \quad x\in\mathbb{R},
\end{equation}
has the Weierstass function $w_{a,b}$ as its real part.\ In \cite{PT}, among other things,  
it is  shown that  $w_{a,b,\mathbb{C}}$
can be modified and randomized to approximate a FBM.\ We suggest \cite{BL} and references quoted for more details of the Weierstrass-Mandelbrot curve and its fractal properties.

The CNDF class also have important topological properties in $C([a,b])$. In fact, as proved independently by Banach and Mazurkiewicz  (\cite{Oxtoby}),  it follows a consequence of Baire category theorem  that the class of  CNDF  on $[a,b]$ is of the second category in $C([a,b])$.
\ Here, as usual, $C([a,b])$ stands for the normed real vector space of  real-valued continuous functions on $[a, b]$ with the supremum norm. 

Let $I=[-1,1]$. We consider here $W: I\times I\longrightarrow \mathbb{R}$, 
the `Weierstrass Fractal kernel', 
defined for $0<a<1$, $b$ an integer such that $ab\geq 1$, and given by
\begin{equation}\label{WK}
W(x,y):=w_{a,b}(x-y), \quad x,y\in I,
\end{equation}
where $w_{a,b}$ is the Weierstrass function (\ref{1.1}).\ This is a continuous, nowhere differentiable, symmetric and positive definite kernel. \ The theory of RKHS tell us  that 
there exists only one RKHS $\mathcal{H}_W :=\mathcal{H}_W (I)$  having the Weierstrass Fractal kernel as reproducing kernel.\ As  it is detailed on the next section,  we present a  complete characterization of the space $\mathcal{H}_W$, given  in terms of Fourier series expansions.\ 

It is 
expected that the functions in a RKHS 
inherit
some of the properties of the generating kernel 
such as  smoothness.
\ In \cite{jordao1}
the authors showed that functions in the RKHS associated to a  kernel having a Mercer series expansion (it means a Fourier-like series given in terms of tensorial product of orthonormal functions) have the same ``degree" of smoothness in terms of the fractional Laplace-Beltrami derivative of the original kernel.\ More general than that, in \cite{jordao2} similar harmonic analysis structure was enough to prove similar reproducing results on a very general context.\ 

\ Reproducing kernel Hilbert spaces (RKHS) are in the formulation of many problems in Approximation Theory, Signal
Analysis, Learning Theory, Geomathematics, etc.\ They comprise an interesting, relevant and useful class of
function spaces in these areas of mathematics and many others.\ The metric structure they carry is an important tool in theoretical and practical aspects, prevailing
the manner in the implementation of procedures in the problems where
they appear.\ We mention references \cite{berlinet,SC,CZ,saitoh,saitoh1}
for a general discussion on RKHS, including potential applications.

 In this paper we present upper and lower estimates for the covering numbers of the embedding  $I_W : \mathcal{H}_W \rightarrow C(I)$  
achieving tight bounds.\ The technique 
applied is similar to the one applied in \cite{Azevedo}
for the covering  numbers of the embedding of the unit ball of the RKHS associated to isotropic kernels on compact two-point homogeneous space.  In short,  here, the approach 
is mainly based on operator norm estimate 
of $I_{K}$ and some others  related finite rank operators. 
In \cite{Kuhn} this approach was employed to obtain estimates for the covering numbers of the embedding operator over the RKHS associated to  the  Gaussian kernel over non-empty interior sets of $\mathbb{R}^d$.
\ The generality of the concept of metric entropy, includes covering number, and it has several applications in many branches of Mathematics, e.g., probability (\cite{Li}), PDEs (\cite{Edmunds}) and operator theory (\cite{Konig}).\ We highlight the important role played by covering number in statistical learning theory to estimate the probabilistic error and the number of samples required for a given confidence level and a given error bound (\cite{SC}).\ We suggest \cite{StCh,zhou1} and references therein for more information on metric entropy and machine learning methods.\ Nevertheless, only for a few infinite-dimensional spaces there has been success in  determining the precise asymptotics of  covering numbers.

If $A$ is a  subset of a metric space $M$ and $\epsilon>0$, the covering number 
\begin{equation}\label{covnum}
\mathcal{C} (\epsilon, A) := \mathcal{C} (\epsilon,A,M)
\end{equation}
is the minimal number of balls in $M$ of radius $\epsilon$ covering the set $A$.\ Clearly, $\mathcal{C}(\epsilon, A) <\infty$, whenever $A$ is a compact subset of $M$.\ For  $X,Y$ Banach spaces and $T: X \rightarrow Y$ an operator the covering numbers are defined in terms of unit balls as follows.\ For $\epsilon>0$, if $B_X$ and $B_Y$ are the unit balls in X and $Y$, respectively, then the covering numbers of $T$ are 
\begin{equation}\label{covnumop}
\mathcal{C} (\epsilon, T):= \mathcal{C}(\epsilon,T(B_X ), Y),
\end{equation}
and given by
 \[
 \mathcal{C} (\epsilon, T)
= \textrm{min} \left\{ n\in \mathbb{N} : \exists y_1, y_2,...,y_n \in Y \,\,\, \mbox{s.t.}\,\,\, T (B_X) \subset \bigcup_{j=1}^{n} (y_j + \epsilon B_Y ) \right\} .\]
Our main goal is to  investigate   the covering numbers of the  embedding $I_W : \mathcal{H}_W \rightarrow {C}(I)$, where
$\mathcal{H}_W$ is the reproducing kernel Hilbert space associated to $W$ defined in equation (\ref{WK}).\ The statement of the main result of this paper will employ the following notations.\ 

Throughout the paper, for functions $f,g: (0,\infty) \rightarrow \mathbb{R}$   we use 
the 
standard asymptotic notations 
\[f(\epsilon) \approx g(\epsilon)\quad \mbox{means}\quad \lim_{\epsilon\rightarrow 0} \frac{f(\epsilon)}{g(\epsilon)}=1\]
and \[ f(\epsilon)\asymp g(\epsilon) \quad \mbox{stands for }\quad 0< \lim_{\epsilon\rightarrow 0} \inf \frac{f(\epsilon)}{g(\epsilon)} \leq \lim_{\epsilon\rightarrow 0} \sup \frac{f(\epsilon)}{g(\epsilon)} < \infty.\]
In particular, for sequences $\{a_k \}$ and $\{b_k \}$ the asymptotic notation $a_k \asymp b_k $ indicate that there is $M,N\geq 0$ such that  $M b_k \leq a_k \leq N b_k$, for all $k$. 

The main result of this paper reads as follows.

 \begin{thm}\label{coveringesatimates} Let $W$ as in \eqref{WK}.   
The covering numbers of the embedding $I_{W}:\mathcal{H}_W \rightarrow C(I)$ behave asymptotically as follows 
 \[\ln(\mathcal{C}(\epsilon, I_W ))\asymp \frac{\left[ \ln \left(1/\epsilon (1-a)^{1/2}\right) \right]^2}{\ln\left(1/a\right)},  \quad \mbox{as}\quad \epsilon \rightarrow 0^+.\]
\end{thm}

The paper is organized as follows.\ In Section \ref{WKer} we introduce the Weierstrass fractal kernel and derive some properties.\ In Section \ref{WRKHS} we give the characterization of the functions in RKHS associated to $W$ in trigonometric series expansions.\ Section 4 is subdivided in three subsections, for the upper bounds, for the lower bounds and proof of the main theorem, respectively.\ We close the paper with some concluding remarks in Section \ref{CRemarks}.

\section{The Weierstrass fractal kernel}
\label{WKer}
In this section we bring some necessary information and properties about the Weierstrass fractal kernel defined on $\mathbb{R}\times \mathbb{R}$.\ We consider $W$ as in formula (\ref{WK}) but for $x,y\in\mathbb{R}$ and we present some properties of the kernel that will be useful to obtain the estimates of the covering numbers of the RKHS associate to $W$.\  

Let $X$ be a non-empty set.\ A symmetric function $K: X \times X \rightarrow \mathbb{R}$, is called a positive-definite kernel 
if it satisfies 
$$\sum_{i=1}^{n}\sum_{j=1}^{n} c_i c_j K(x_i,x_j ) \geq 0$$
for all $n\geq 2, \{ x_1 , x_2 ,...,x_n \}\subset X$ and $\{c_1 ,c_2 ,...,c_n \}\subset \mathbb{R}$.\ For a symmetric positive-definite kernel $K:X \times X \rightarrow \mathbb{R}$ there exists a (unique) Hilbert space 
$$
\mathcal{H}_K:=\left( \mathcal{H}_K (X) , \langle \,\cdot \, ,\,\cdot \, \rangle_K \right)
$$
of real valuable function on $X$, such that the point evaluation linear functional, i.e.,  
$$
I_x (f) := f(x), \quad f\in \mathcal{H}_K
$$
for each $x\in \mathbb{R}$, is continuous.\ The Riez representation theorem implies that $I_x$  is characterized as follows 
$$
I_x (f)= \langle f, K(x, \,\cdot \,)\rangle_K, \quad f\in \mathcal{H}_K, \quad x\in \mathbb{R}. 
$$
In particular, $\mathcal{H}_K$ is the closure of $\spn\{ K(y, \,\cdot \,) : y \in I\}$.\ All these information and more details about the reproducing kernel Hilbert spaces 
(also known as Native spaces) can be found in \cite{aron,CZ} and others quoted there for a general discussion.   

\begin{lem}\label{LEMAPDW}
The function $W: \mathbb{R}\times \mathbb{R}\rightarrow \mathbb{R}$ is a continuous and positive definite kernel.
\end{lem}


\begin{proof}
Since $0<a<1$ the trigonometric series defining $W$ is uniformly convergent and it is clear that $W$ is continuous.\ In order to show that $W$ is a positive definite kernel on $\mathbb{R}\times \mathbb{R}$ we observe that for any $k>1$ integer, and $c_1, \ldots , c_k$ and 
 $x_1,x_2,\cdots,x_k $ real numbers, we have that
\begin{align*}
\sum_{i=1}^{k}\sum_{j=1}^{k}c_{i}c_{j}W(x_{i},x_{j})&= 
\sum_{i=1}^{k}\sum_{j=1}^{k}c_{i}c_{j}\sum_{n=0}^{\infty} a^n \cos(b^n \pi (x_i -y_j )),
\end{align*}
ans this last summation is:
\begin{align*}
\sum_{n=0}^{\infty} a^n \left\{\left(\sum_{i=1}^{k}c_{i}\cos(b^n \pi x_{i})\right)^2 +\left(\sum_{i=1}^{k}c_{i}\sin(b^n \pi x_{i})\right)^2\right\}.
\end{align*}
Therefore $\sum_{i=1}^{k}\sum_{j=1}^{k}c_{i}c_{j}W(x_{i},x_{j}) \geq 0$ and $W$ is positive definite.
\end{proof}

First, we begin by showing that $W$ is  continuous and nowhere differentiable kernel.

\begin{lem} \label{L1}
The Weierstrass fractal kernel on $\mathbb{R}\times \mathbb{R}$ is nowhere differentiable.
\end{lem}

The nowhere differentiability of $w$ as in formula (\ref{1.1}), proved by Hardy in \cite{hardy}, implies that the partial derivatives of $W$ do not exist and, then $W$ is also nowhere differentiability on $\mathbb{R}\times \mathbb{R}$.

\begin{lem}\label{WCS} For all $x,y\in \mathbb{R}$ we have
$$W(x,y)= C(x,y)+S(x,y),$$
where $C$ and $S$ are the following positive definite kernels  
 $$C(x,y)= \sum_{n=0}^ \infty a^n \cos(b^n \pi x) \cos(b^n \pi y )$$ and
 $$S(x,y)=\sum_{n=0}^ \infty a^n \sin(b^n \pi x) \sin(b^n \pi y).$$
\end{lem}

\begin{proof}  
The proof follows directly from the standard trigonometric identity for the cosine function.\ The positive definiteness of $C$ and $S$ follows from the same arguments in the proof Lemma \ref{LEMAPDW}.
\end{proof}
\section{The RKHS associated to $W$}
\label{WRKHS}
In this section we will apply the previous results in order to prove a characterization of the RKHS associated to the kernel $W$ and to derive some properties.\ From now on we consider the Weierstrass fractal kernel defined on $I \times I$, however all the results we present can be stated, with some mild adaptations,  using $I=[-r,r]$, with $r>0$ integer.  

Since $W: I \times I\to \mathbb{R}$ is a positive definite kernel (Lemma \ref{LEMAPDW}), there is a  (unique)  RKHS $\mathcal{H}_W$ induced by $W$.\ It is naturally endowed with the inner product 
$\langle{\,\cdot \,},{\, \cdot \,}\rangle_{W}:\mathcal{H}_W \times \mathcal{H}_W \to \mathbb{R}$, satisfying: 

\textbf{R1.} The function $y\mapsto W(x,\cdot)$ belongs to $\mathcal{ H}_W$, for all $x\in I$;

\textbf{R2.} (Reproduction property): $f(x)=\langle  f,W(x,\cdot) \rangle_{W}$, for all $f\in \mathcal{ H}_W$ and $x\in I$.  

The next result provides detailed information about the elements of $\mathcal{ H}_W:=\mathcal{ H}_W (I)$, where $I=[-1,1]$.
Here $\ell^2$ stands for the usual Hilbert space of the square summable sequences.

\begin{thm}\label{teo1} The RKHS induced by the Weierstrass fractal kernel \eqref{WK} is $\mathcal{H}_W$ defined as the following space of functions
\begin{equation*}
\left\{ f(x) = \sum_{n=0}^{\infty} a^{n/2} \left[c_n \cos(b^n \pi x)+  d_n \sin(b^n \pi x)\right] : (c_n) ,(d_n)\in \ell^2,\,x\in I \right\},
\end{equation*}
endowed with the inner product
$$
\langle  f,g \rangle_{W} :=\sum_{n=0}^{\infty}\left(c_n e_n +d_n f_n\right), \quad f,g\in \mathcal{H}_W,
$$
where
$$f(x)= \sum_{n=0}^{\infty} a^{n/2} \left[c_n \cos(b^n \pi x)+ d_n \sin(b^n \pi x)\right]$$ and $$g(x)= \sum_{n=0}^{\infty} a^{n/2} \left[e_n \cos(b^n \pi x)+ f_n \sin(b^n \pi x)\right],
$$ 
for $x\in I$

\end{thm}

\begin{proof} It is clear that $(\mathcal{H}_W, \langle  \, \cdot \, ,\, \cdot \, \rangle_{W})$ is itself a Hilbert space.\ Due the unicity of the RKHS, we will show that $\mathcal{H}_W$ is the RKHS induced by $W$ verifying that it fits in properties R1 and R2 above. 

For $n\geq 0$ an integer, we consider 
$$
c_n(x):={a^{n/2} \cos(b^n \pi x)}\quad \mbox{and} \quad d_n(x):={a^{n/2} \sin(b^n \pi x)},
$$
for $x\in I$.\ For each $x\in I$, the sequences $\{c_n(x)\}, \{d_n(x)\}$ belongs to $\ell^2$ and, then
$$
W(x, \, \cdot \,) = \sum_{n=0}^\infty a^{n/2} \left[ c_n(x)  \cos(b^n \pi \,\, \cdot \,\,) + d_n(x) \sin(b^n \pi \,\, \cdot \,\,)\right]
$$
belongs to $\mathcal{H}_W$.\ Thus,  span$\{W(x,\, \cdot \,),\,x\in I\}$ is a subspace $\mathcal{H}_W$ and property R1 follows. 

For every $f\in \mathcal{H}_W$, we write  
$$f(x) = \sum_{n=0}^{\infty} a^{n/2} \left[c_n \cos(b^n\pi x) +  d_n \sin(b^n\pi x)\right],$$
with $x\in I$ 
and  clearly
$$
\langle f, W(x, \cdot ) \rangle_{W} 
=f(x), \quad x\in I.
 $$
 This also implies that $\spn \{W(x,\cdot),\,\,x\in I\}$ is a dense subspace of $\mathcal{H}_W$.\ And property R2 is proved.\ 
 \end{proof}

Property R2 above implies that \begin{align*}
     \langle W(x, \cdot ), W(y, \cdot ) \rangle_{W} 
     &= W(x,y)
 \end{align*}
 for every $x,y\in I$.\ As proved above, the RKHS $\mathcal{H}_W$ associated to $W$ is defined as the closure of $\spn \{W (x,\cdot) : x\in I \}$, endowed with the inner product satisfying  
$$\left|\left| \sum_{i=1}^{m} c_i W(x_i ,\cdot)\right|\right|_W = \sum_{i=1}^{m} \sum_{j=1}^{m} c_i c_j W (x_i,x_j), $$
for all $m\in\mathbb{N}$ and $x_1, \ldots, x_m\in I.$
 
Let $L^2 (I):= L^2 (I,dx) $ be the usual Hilbert space of integrable functions with the usual inner product 
$$
\langle f,g  \rangle_2 = \int_{-1}^{1} f(x)g(x)\, dx, \quad f,g \in L^2 (I). 
$$
The set 
$$
\{\cos(b^{n}\pi \,\, \cdot \,\, ),\sin(b^{n}\pi \,\, \cdot \,\, ): n\in \mathbb{N}\cup\{0\}\}, 
$$
is an orthonormal set in $L^2(I)$.\ And due the fact that $0<a<1$ the system 
$$
\{a^{n/2}\cos(b^{n}\pi \,\, \cdot \, \,),\, a^{n/2}\sin(b^{n}\pi \,\, \cdot \, \,): n\in \mathbb{N}\cup\{0\}\}, 
$$
is an orthonormal basis of the space $\mathcal{H}_{W}$ (\cite[Lemma 2.6]{SS}).

We also have that $\mathcal{H}_{W} \subset L^2 (I)$ and $\mathcal{H}_{W}=\mathcal{C}\oplus \mathcal{S}$ with $\mathcal{C}$ and $\mathcal{S}$ the following collection of functions on $I$:
$$
\mathcal{C} = \left\{g(x)= \sum_{n=0}^{\infty} a^{n/2} c_n \cos(b^n \pi x) :  \{c_n\} \in \ell ^2  \right\}$$
and 
$$\mathcal{S} = \left\{h(x)=\sum_{n=0}^{\infty} a^{n/2} d_n \sin(b^n \pi x) : \{d_n\} \in \ell ^2  \right\}.$$
 
 It is important to note that Theorem \ref{teo1} gives us a Fourier series-like characterization of the elements of $\mathcal{H}_{W}$.\ This is somehow expected, since $W$ itself has this same type of structure (see \cite{SS}).\ Along with this reasoning we have a characterization for functions in $\mathcal{H}_{W}$ in terms of  nowhere differentiability.\ We will apply the next lemma.

\begin{lem}\label{JO} \cite[p. 22]{johsen} Let $f:\mathbb{R}\longrightarrow  \mathbb{C}$ given by
$$
f(x)=\sum_{n=0}^{\infty}a_n\exp(ib_n (x))
$$
for a complex sequence $\{a_n\}$ having the norm series of its elements summable and $\{b_n\}$ with positive elements such that $b_n\nearrow \infty$.\ If
$$
\liminf_{n\to\infty} b_{n+1}/b_n>1 \quad \mbox{and} \quad \lim_{n\to\infty} a_nb_n\neq 0,
$$
then $f$ is nowhere differentiable.\ If $\sup_n\{|a_n| b_n\}=\infty$, then $f$ is not Lipschitz continuous at any point.\
\end{lem}

\begin{thm} Consider $f\in\mathcal{H}_W$ with series expansion 
\begin{equation}\label{finwk}
f(x)= \sum_{n=0}^{\infty} a^{n/2} \left[c_n\cos(b^n \pi x)+ d_n \sin(b^n \pi x)\right].
\end{equation}
If
\begin{equation}\label{limfinwk}
\lim_{n\rightarrow\infty}c_n (a^{1/2} b)^n \neq 0 \quad
\textrm{and} \quad \lim_{n\rightarrow\infty} d_n (a^{1/2} b)^n \neq 0,
\end{equation} 
then $f$ is nowhere differentiable.\ Moreover, if $H$ is the collection of functions in $\mathcal{H}_{W}$ of the form (\ref{finwk}) satisfying (\ref{limfinwk}), then 
\[\overline{H}=\mathcal{H}_{W},\] 
with respect to the norm $\Vert \cdot \Vert_W$.
\end{thm}

\begin{proof}
First, we  prove that $H$ is dense in $\mathcal{H}_{W}$.\ So, for any
\[
f(x)= \sum_{n=0}^{\infty} a^{n/2}\left[ c_n \cos(b^n \pi x))+d_n \sin(b^n \pi x)\right]
\]
in $\mathcal{H}_W$, we consider the sequence $\{f_N \}$ defined by 
\[
f_N (x)= \sum_{n=0}^{\infty} u_n^N a^{n/2} \cos(b^n \pi x) +v_n^N a^{n/2}\sin(b^n \pi x), \quad x\in I,
\]
where
\[u_n^N = \begin{cases}

              a^{n/2} , & \mbox{if } n\geq N  \\

             c_n, & \mbox{if}\hspace{0.1cm} 1\leq n \leq N-1 

       \end{cases} \quad\textrm{and}\quad v_n^N = \begin{cases}

              a^{n/2} , & \mbox{if } n\geq N  \\

             d_n, & \mbox{if}\hspace{0.1cm} 1\leq n \leq N-1.

       \end{cases}
 \]

It is clear that $\{ f_N \}\subset \mathcal{H}_W$.\ Since $ab\geq 1,$ we have that 
\begin{eqnarray*}
\lim_{n\rightarrow\infty} u_n^N a^{n/2} b^n & = & \lim_{n\rightarrow\infty} v_n^N a^{n/2} b^n \\ & = & \lim_{n\rightarrow\infty} a^n b^n =\infty,
\end{eqnarray*}
and we conclude that $\{ f_N \}\subset H$.\ Finally, the orthonormality of 
\[\{ a^{n/2} \cos(b^n \pi\,\, \cdot \,\,), a^{n/2} \sin(b^n \pi \,\, \cdot \,\,) : n\in \mathbb{N}\cup\{0\}\}\] 
in $\mathcal{H}_W$, implies that 
$\lim_{N\rightarrow \infty} \Vert f-f_N \Vert_W =0$.\ 
From an application of Lemma \ref{JO} it follows that 
$H$ is composed of CNDF. 

\end{proof}
\section{Estimates for the covering numbers}
We proceed by splitting this section in three subsections in order to prove the estimates for the upper (Subsection 3.1) and lower bounds (Subsection 3.2) for the covering numbers for $\mathcal{C}(\epsilon, I_W )$.\ In Subsection 3.3 we present the proof of Theorem \ref{coveringesatimates}.\ The covering number was defined in equation (\ref{covnumop}), Section \ref{intro}.\   

Our framework implies directly the norm estimates
$$\|f\|_{2}\leq \|f\|_{p}\leq \|f\|_{\infty}.$$
Hence,  it is enough to prove the upper
bound for the covering numbers with respect to the sup-norm and the lower bound with respect to
the L2-norm.
\subsection{Upper bounds}
 
The next result presents some properties of covering numbers.\ The proofs are along the same lines as the proofs for the corresponding properties of entropy numbers in \cite[Lemma 1]{Kuhn} and \cite[Section 2.d]{Konig}.
 
 \begin{lem} \label{LemCN}
Consider $S, T: X\rightarrow Y$ and $R: Z \rightarrow X$ operators on (real) Banach spaces.\ For any $\epsilon, \delta > 0$ the following holds: 
 
\begin{enumerate}
 \item $\mathcal {C} (\epsilon + \delta, T+S) \leq \mathcal{C} (\epsilon, T) \hspace{0.1cm} \mathcal{C}(\delta, S)$.

 \smallskip
   
 \item If $rank(T) <\infty$, then $$\mathcal{C}(\epsilon, T)\leq \left( 1+2\| T\|/\epsilon \right)^ {rank (T)}.$$ 
 
 \item If $\| T\| \leq \epsilon $, then $\mathcal{C}(\epsilon, T)=1$.
 
 \smallskip
 
\item For any $\epsilon_1<\epsilon_2$ it holds $\mathcal{C}(\epsilon_2,T)<\mathcal{C}(\epsilon_1, T)$.

\smallskip

\item  $\mathcal{C}(\epsilon \delta,TR)\leq \mathcal{C}(\epsilon, T)\hspace{0.1cm}\mathcal{C}(\delta, R)$.
\end{enumerate}
 \end{lem}
 
 
 \begin{prop}
 The embedding $I_W : \mathcal{H}_W \longrightarrow C(I)$ satisfies 
 $$
  \| I_W \|^2  = \frac{1}{1-a}.
  $$
 \end{prop}
 
\begin{proof}
 By the reproducing property R2, Section \ref{WRKHS}, and an application of the Cauchy-Schwarz inequality we have that
\begin{align*}
\| I_W  \|^2 
       = \sup_{\| f \|_W =1} \sup_{x\in I} |f(x)| ^2 
     \leq 
  \sup _{x\in I} \| W(x,\cdot)\|_W  ^2  
   =  \frac{1}{1-a}.
\end{align*}
If we consider $f\in \mathcal{H}_W$,
$$
f(x)= \sum_{n=0}^\infty a^{n/2} \left[c_n \cos(b^n \pi x)) + d_n \sin(b^n \pi x)\right]
$$
then $g(x)= f(x)/\| f\| _W$, $x\in I$, is such that $g\in\mathcal{H}_W$ and $\|g\|_W=1$.\ For $f$ as above and $c_n =a^{n/2}$, $n\in\mathbb{N}$, we have that $g$ at $x=0$ attains the upper bound $1/(1-a)$ and the proof follows. 
\end{proof}

Proposition above can be seen as special case of a general case of Lemma 4.23 in \cite{StCh}.\ If  $\mathcal{H}_{K}(X)$ is the reproducing kernel Hilbert space generated by a bounded kernel $K$ on $X$, then 
$$\| I_K:\mathcal{H}_{K}(X) \to \ell_{\infty}(X) \|^2 = \sup_{x\in X}K(x,x).$$

We will need a similar result for certain projection operators, defined as follows.\ Given $N>1$ integer we define $P_{\mathcal{U}}^N$ and $P_{\mathcal{V}}^{N}$ as being the projections onto 
\begin{align*}
\mathcal{U}:=\textrm{span}&\, \left\{ a^ {n/2} \cos (b^n \pi \,\,\cdot \,\,) , a^ {n/2} \sin(b^n \pi\,\,\cdot \,\,) : n= 0, \ldots ,N-1\right\} 
\end{align*}
and
\begin{align*}
\mathcal{V}:=\textrm{span}& \left\{ a^ {n/2} cos (b^n \pi\,\,\cdot \,\,) , a^ {n/2}  \sin(b^n \pi\,\,\cdot \,\,) :  n= N, \ldots ,\infty\right\}  \end{align*}
respectively. 

Precisely, if $g\in \mathcal{U}$, then 
$$g(x)=\sum_{n=0}^{N-1}c_{n}(g)a^{n/2}\cos(b^n \pi x)
 + d_{n}(g)a^{n/2}\sin(b^n \pi x)$$
for some real numbers $c_{n}(g)$ and $d_{n}(g)$, for $n\geq 0$.\ And similarly for the space $\mathcal{V}$.  
Thus,  for any $f\in\mathcal{H}_{W}$, we have 
 $$
 f=P_{\mathcal{U}}^N(f) + P_{\mathcal{V}}^{N}(f).
 $$ 

\begin{lem}\label{projnorm} If $P_{\mathcal{U}}^N$ and $P_{\mathcal{V}}^{N}$ are the projections defined above, then 
\[\| I_W P_{\mathcal{U}}^N \|^2 = \frac{1- a^N}{1-a} \quad \mbox{and} \quad \| I_W P_{\mathcal{V}}^{N} \|^2 = \frac{a^N}{1-a}.\]
\end{lem}

\begin{proof} We first observe that 
$$
\| I_W P_{\mathcal{U}}^N \|^2  =  \sup_{\| g \|_W =1} \sup_{x\in I} | g(x) | ^2,
$$
and the right-hand side o equality above is upper bounded by
\begin{align*}
                   \sup_{\| g \| _W  =1} \sup _{x\in I}\left(\sum_{n=0}^{N-1} {c_n(g)}^2 +{d_n(g)}^2 \right) \| W_N (x,\cdot)\|_W  ^2
                   &  \leq 
                   \sum_{n=0}^ {N-1} a^n.
\end{align*}
Thus, 
\begin{align*}
\| I_W P_{\mathcal{U}}^N \|^2 & =  \sup_{\| g \|_W =1} \sup_{x\in I} | g(x) | ^2
                    \leq 
                     \frac{1- a^N}{1-a}.
\end{align*}
Analogously as proof above, if we consider $\widetilde{g}(x) =g(x)/\| g\|_W$, $x\in I$, with $g\in \mathcal{U}$ and $c_n(g) = a^{n/2}$, then 
$$
[\widetilde{g} (0)]^2 = \frac{1- a^N }{1-a} \leq {\| I_W P_{\mathcal{U}}^N \|}^2 \leq \frac{1- a^N}{1-a},
$$
and we finish the proof of 
$$
\| I_W P_{\mathcal{U}}^N \|^2 = \frac{1- a^N}{1-a}.
$$

The proof of the another equality is completely analogous to the one we just presented the proof and we choose to omit it here. 
\end{proof} 

From now on, for each $N \geq 1$, we will employ the notation 
$$\mu_N := \left( \frac{a^N}{1-a}\right)^{1/2}.$$

\begin{thm}\label{tma4.4}
For any $\epsilon > 0$, there exist a positive constant $M_1$ such that  $$\ln (\mathcal{C} (2\epsilon, I_W))\leq  \frac{M_1}{\ln(1/a)}\left\{  \ln\left[ \left(\frac{1}{1-a}\right)^{1/2}\frac{1}{\epsilon}\right]\right\}^2  .$$
\end{thm}

\begin{proof} Since $\mu^2_N = a^N/(1-a)$, for each $N\geq 1$, we have that  $(1-a) \mu_N^2 = a^{N}$  and that 
\begin{equation} \label{4.5}
    \frac{\ln (  (1-a) \mu_N^2 )}{\ln a} = N.
\end{equation} 
Now, we consider $\epsilon >0$ and observe that there exists a natural number $N=N(\epsilon)$ such that 
\begin{equation}\label{MuNEpsilon}
    \mu_N \leq \epsilon< \mu_{N-1}.
\end{equation}
If $\epsilon \rightarrow 0^+$,  then $\mu_N \rightarrow 0$ and it follows that

\begin{equation}N(\epsilon) \approx \frac{\ln({\epsilon}^2 (1-a))}{\ln a}.
\end{equation}
From Lemma \ref{projnorm} we have $\| I_W P_{\mathcal{V}}^{N} \| = \mu_N$ and an application of Lemma \ref{LemCN}, item 3, we obtain 
$$
\mathcal{C}(\mu_N ,I_W P_{\mathcal{V}}^{N} )=1.
$$
Lemma \ref{LemCN}, item 2, implies that 
\begin{eqnarray*} 
\mathcal{C} (\epsilon, I_W P_{\mathcal{U}}^N ) & \leq & \left(1+ \frac{2\| I_W P_{\mathcal{U}}^N \| }{ \epsilon} \right)^ {rank(I_W P_{\mathcal{U}}^N )}\\ & \leq & \left( 1+\frac{{2}}{ {\epsilon }} \left( \frac{{1-a}^N}  {1-a} \right)^{1/2}\right)^N .
\end{eqnarray*}
By the formula (\ref{MuNEpsilon}), we obtain that
\begin{eqnarray*} 
 \left( 1+\frac{{2}}{ {\epsilon }} \left( \frac{{1-a}^N}  {1-a} \right)^{1/2}\right)^N & = & \left[ 1+ \frac{2}{\epsilon} \left( \frac{1}{1-a} - {{{ \mu_N ^2 }}}\right)^{1/2} \right]^ N\\ 
                             & < & \left[ 1+ \frac{2}{\epsilon} \left( \frac{1}{1-a} \right)^{1/2} \right]^ N. 
\end{eqnarray*}
Combining these estimates, due Lemma \ref{LemCN}, we have
\begin{eqnarray*} 
\mathcal{C}(2\epsilon, I_W  ) \leq \mathcal{C} (\epsilon, I_W P_{\mathcal{U}}^N ) <  \left[ 1+ \frac{2}{\epsilon} \left( \frac{1}{1-a} \right)^{1/2} \right]^ N.
\end{eqnarray*}
Thus, 
\begin{align*}
    \ln(\mathcal{C} (2\epsilon , I_W ))  & < N  \ln \left[1+ \frac{2}{\epsilon} \left( \frac{1}{1-a} \right)^{1/2} \right] \\
      & \leq  \frac{2M}{\ln{(1/a)}} \ln \left( \frac{1}{\epsilon (1-a)^{1/2}}\right)   \ln \left( \frac{4}{\epsilon ({1-a})^{1/2}} \right),
    \end{align*}
for some constant $M>0$.\ 
Hence, for every $\epsilon>0$ sufficiently small it follows the estimate
\begin{align*}
\ln (\mathcal{C}(2\epsilon, I_W)) &\leq 
    \frac{M_1}{\ln (1/a) } \left[ \ln \left( \frac{1}{\epsilon (1-a)^{1/2}}\right) \right]^2, 
\end{align*}
and we finish the proof.
\end{proof}
\subsection{Lower bounds}

In order to  present lower bounds for 
$\mathcal{C} (\epsilon, I_K )$, we first introduce the following  general result.

 \begin{lem}\label{CN1}\cite[Lemma 1]{Kuhn}
Let $X$ and $Y$ two n-dimensional real Banach spaces and  $T: X  \rightarrow Y$ be an operator. If  $\epsilon>0 $ then  
\begin{equation}
\mathcal{C}(\epsilon,T) \geq |\det \sqrt{T^{*}T}| \left(\frac{1}{\epsilon}\right) ^n.
\end{equation}
\end{lem}






A lower bound for $\mathcal{C} (\epsilon, I_W)$ is presented in the following theorem.

\begin{thm}\label{LB}
For any $\epsilon > 0$, there exist a positive constant $M_2$ such that  
$$
\ln (\mathcal{C} (\epsilon, I_W))\geq  \frac{M_2}{\ln\left(1/a\right)}\ln\left(1/(1-a)^{1/2}\epsilon\right) \ln\left( \frac{1}{\epsilon}\right).
$$
\end{thm}

\begin{proof}

Let $E:= \{  \psi _k : k\in \mathbb{N}\cup \{0\}\}$ be the orthogonal basis of $\mathcal{H}_W$, given by 
$\{ \cos(b^n\pi(\, \cdot \,)), \sin(b^n\pi(\, \cdot \,)),\,\,n\in\mathbb{N}\cup \{0\}\}$, as indicated in  Theorem \ref{teo1}. \ Also, consider for $n\in\mathbb{N}$ the operator, defined by the following composition
$$
T_n : E_n \stackrel{ J_n }{\longrightarrow} \mathcal{H}_W  \stackrel{I_W}{\longrightarrow} L^2 ( I) \stackrel{P_n }{\longrightarrow} F_n,
$$
where $J_n$ is the embedding of $E_n :=$ span $\{\psi_k : k=0,..,n-1 \}$
 into $ \mathcal{H}_W $ and
$P_n $ is the orthogonal projection on $F_n := I_W(J_n(E_n))$.\ The operator $T_n :E_n \rightarrow F_n $ is uniquely determined by the following relation 
$$
T_n \psi_k (x) = \psi_k (x),\quad k=0,1,\ldots,n-1.
$$  

If $A_n$
is the representing matrix of the operator $T_n ^* T_n: E_n \rightarrow E_n $  with respect to the ONB $\{ \psi_0 , \psi_1 ,...,\psi_{n-1} \}$ of $E_n $, then $$\det A_n = 1, \quad n\geq 1.$$



From Lemma \ref{LemCN}, for any $\epsilon>0$, the following holds 
\begin{eqnarray*}
\mathcal{C} (\epsilon, T_n )=\mathcal{C} (\epsilon, P_n  I_W J_n ) & \leq & \mathcal{C} (1, P_n )\mathcal{C} (\epsilon, I_W )\mathcal{C} (1, J_ n )\\ & = & \mathcal{C} (\epsilon, I_W ).
\end{eqnarray*}
Now, Lemma \ref{CN1} implies
$$\mathcal{C} (\epsilon, I_W ) \geq \mathcal{C} (\epsilon, T_n  )\geq \sqrt{ \det (T_n ^* T_n )}\left(\frac{1}{\epsilon}\right)^n. 
$$
Finally, we observe that for $\epsilon>0$ sufficiently small, we can choose 
$$n\geq n_{\epsilon,a}= [\ln({\epsilon}^2 (1-a))]/\ln a,
$$ 
and we obtain,
 \begin{align*}
 \ln (\mathcal{C}(\epsilon, I_W ) &\geq 
  \frac{M_2}{\ln(1/a)} {\ln\left(\frac{1}{(1-a)^{1/2}\epsilon }\right)} \ln\left( \frac{1}{\epsilon}\right).
  \end{align*} for some constant $M_2 >0$.
\end{proof}

\subsection{Proof of Theorem \ref{coveringesatimates}}

We close this section presenting the proof of our main result.

\begin{proof} We first observe that 
$$
\ln\left( \frac{1}{\epsilon}\right) \approx  \ln\left(\frac{1}{(1-a)^{1/2}\epsilon }\right),  \quad\mbox{as} \quad \epsilon\to 0^{+}.
$$
Theorem \ref{LB} implies 

\[ \frac{M_2}{\ln(1/a)} \left\{{\ln\left(\frac{1}{(1-a)^{1/2}\epsilon }\right)} \right\}^2 \leq \ln (\mathcal{C} (\epsilon, I_W)),\]
for some $M_{2}>0$.\ Since \[ \ln (\mathcal{C} (\epsilon, I_W)) \asymp  \ln (\mathcal{C} (2\epsilon, I_W)),\quad \mbox{as} \quad  \epsilon\rightarrow 0^+,\] 
an application of Theorem \ref{tma4.4} leads us to both
\begin{eqnarray*} 
\frac{M_2}{\ln(1/a)} \left\{{\ln\left(\frac{1}{(1-a)^{1/2}\epsilon }\right)} \right\}^2 & \leq & \ln (\mathcal{C} (\epsilon, I_W)),
\end{eqnarray*}
and
\begin{eqnarray*} 
\ln (\mathcal{C} (\epsilon, I_W)) & \leq & \frac{M_1}{\ln(1/a)}\left\{  \ln\left[ \left(\frac{1}{1-a}\right)^{1/2}\frac{1}{\epsilon}\right]\right\}^2,
\end{eqnarray*}
and the proof follows.\end{proof}

\section{Concluding remarks}\label{CRemarks}

In this paper, we have introduced  and investigated the  Weierstrass continuous nowhere differentiable kernel $W$ given in \eqref{WK}. The characterization of the RKHS $\mathcal{H}_{W}$ induced by $W$,   presented in Theorem \ref{teo1},    shows that the elements have a  Fourier-like expansion nature inherited by the  kernel alongside  the nowhere differentiability, which is manifested for 
a dense subset of  this space. 

Even though the  notion of $\epsilon$-entropy (covering numbers, for instance) play an important role in several areas,
exact asymptotic behaviour of such  entities has been successfully determined only for a few infinite-dimensional spaces. Finally, the (sharp) estimates obtained for the covering numbers in this paper are  independent of the parameter $b$, which, somehow,  indicates  that  the periodicity  of  $W$ does not  affect metric entropy related to $\mathcal{H}_{W}$. 


%
%



\end{document}